\documentclass[12pt]{amsart}
\usepackage{amsmath}
\usepackage{amssymb}
\usepackage{amscd}
\usepackage{graphicx}
\usepackage{calrsfs}
\usepackage{enumerate}
\usepackage{epigraph}

\usepackage{amsthm}
\usepackage{xy}

\newtheorem{thm}{Theorem}[section]

\newtheorem{cor}[thm]{Corollary}
\newtheorem{lem}[thm]{Lemma}

\newtheorem{prop}[thm]{Proposition}

\theoremstyle{definition}

\newtheorem{ex}[thm]{Example}

\newtheorem{prob}{Problem}

\theoremstyle{remark}
\newtheorem{rms}[thm]{Remark}

\newcommand{\RR}{\mathbb{R}}
\newcommand{\CC}{\mathbb{C}}
\newcommand{\HH}{\mathbb{H}}
\newcommand{\OO}{\mathbb{O}}

\newcommand{\R}{\mathbb{R}}
\newcommand{\Q}{\mathbb{Q}}

\newcommand{\Z}{\mathbb{Z}}
\newcommand{\I}{\mathbb{I}}

\newcommand{\bH}{\mathbf{H}}
\newcommand{\F}{\mathcal{F}}
\newcommand{\G}{\Gamma}
\newcommand{\wt}{\widetilde}

\newcommand{\nb}{{\mathcal Op}}

\DeclareMathOperator{\id}{id}
\DeclareMathOperator{\Ad}{Ad}
\DeclareMathOperator{\Lie}{Lie}

\DeclareMathOperator{\PSL}{PSL}

\DeclareMathOperator{\SO}{SO}
\DeclareMathOperator{\Sr}{S}
\DeclareMathOperator{\Or}{O}

\DeclareMathOperator{\PU}{PU}
\DeclareMathOperator{\PSO}{PSO}
\DeclareMathOperator{\GA}{GA}
\DeclareMathOperator{\Aut}{Aut}
\DeclareMathOperator{\Isom}{Isom}

\DeclareMathOperator{\hol}{Hol}
\DeclareMathOperator{\dev}{Dev}

\DeclareMathOperator{\codim}{codim}

\DeclareMathOperator{\per}{per}
\def\mun{{^{-1}}}
\DeclareMathOperator{\Diff}{Diff}

\numberwithin{equation}{section}

\usepackage{enumitem}

\newtheorem*{ack}{Acknowledgment}

\usepackage{color}
\definecolor{darkgreen}{cmyk}{1,0,1,.2}
\definecolor{darkorchid}{rgb}{0.6, 0.2, 0.8}
\definecolor{persimmon}{rgb}{0.5, 0.2, 0.5}

\newdimen\theight
\def\TeXref#1{%
             \leavevmode\vadjust{\setbox0=\hbox{{\tt
                     \quad\quad  {\small \textrm #1}}}%
             \theight=\ht0
             \advance\theight by \lineskip
             \kern -\theight \vbox to
             \theight{\rightline{\rlap{\box0}}%
             \vss}%
             }}%

\begin{document}

\title[Rigidity of Lie foliations]{Rigidity of Lie foliations with locally symmetric leaves}
\author{Ga\"{e}l Meigniez and Hiraku Nozawa}

\address{Ga\"{e}l Meigniez, I2M (UMR CNRS 7373) ---
Aix-Marseille Universit\'{e},
I2M (UMR 7373 CNRS),
 3 place Victor-Hugo, Case 19, 13331 Marseille Cedex 3,
France
}
\email{gael.meigniez@univ-amu.fr}

\address{Hiraku Nozawa, Department of Mathematical Sciences, College of Science and Engineering, Ritsumeikan University, 1-1-1 Nojihigashi, Kusatsu, Shiga, 525-8577, Japan}
\email{hnozawa@fc.ritsumei.ac.jp}


\begin{abstract}
We prove that if the leaves of a minimal Lie foliation are locally isometric to a symmetric space of non-compact type without a Poincar\'e disk factor, then the foliation is smoothly conjugate to a homogeneous Lie foliation up to finite covering. {This result generalizes and strengthens Zimmer’s theorem, which characterizes minimal Lie foliations with leaves isometric to a symmetric space of non-compact type without real rank one factors as pullbacks of homogeneous foliations. As applications,} we extend Zimmer’s arithmeticity theorem for holonomy groups and establish a rigidity theorem for Riemannian foliations with locally symmetric leaves.
\end{abstract}

\maketitle

\section{Introduction}\label{intro_sec}

A \emph{Lie} foliation is a foliation transversely modelled on a
 Lie group, called \emph{structural}.
  The interest towards Lie foliations mainly
   comes from P. Molino's theory, \cite{Molino}
   after which every Riemannian foliation is in
   some sense a quotient of a
   parametric family of Lie foliations.
     Yet, little is known about the classification of Lie foliations themselves.
      The first examples are the so-called \emph{homogeneous} Lie foliations,
       which are foliations on locally homogeneous spaces by locally homogeneous leaves.
        The main questions deal with to what extent these examples are representative of the whole: see Ghys \cite{Ghys}.
         Examples abound of Lie foliations whose structural group
         is solvable and that cannot be obtained as 
 {pullbacks} of homogeneous ones
 \cite{Meigniez1,Meigniez2}.

A breakthrough was made by R.\ Zimmer \cite{Zimmer}. He brought methods of rigidity theory to study Lie foliations with symmetric leaves of non-compact type whose factors are of real rank at least two. By means of generalising Margulis' superrigidity theorem to cocycles,
Zimmer showed that the holonomy groups of such foliations are arithmetic.
As  pointed out by E.\ {Ghys}, Zimmer actually showed that such Lie foliations are necessarily {pullbacks} of homogeneous ones.

The present paper more generally studies the Lie foliations whose leaves are \emph{locally} isometric to a symmetric space $X$ of non-compact type \emph{without Poincar\'e disk factor.}
 We put no restriction on the real rank of the factors, nor on the fundamental group of the leaves.
We apply the fundamental methods
  and results of modern rigidity theory. Our main result is that these Lie foliations are \emph{smoothly conjugate} to homogeneous ones {up to finite coverings (Theorem \ref{main1_thm}). As consequences,} we generalise Zimmer's arithmeticity theorem {(Corollary \ref{z_cor})}. Also, with the help of Molino's structure theorem, we get a rigidity property for Riemannian foliations with locally symmetric leaves {(Theorem \ref{uni_thm})}. In Section \ref{prob_sec}, we propose several fundamental problems on Lie foliations that arise from
  our techniques and results.

\begin{ack}
The second author is partially supported by JSPS KAKENHI Grant Numbers 26800047, 17K14195, 20K03620, 24K06723 and the Spanish MICINN grant MTM2011-25656. This work was partially
carried during a fellowship of the second author at Institut des Hautes \'{E}tudes Scientifiques (Bures-sur-Yvette, France), supported by EPDI/\-JSPS/\-IH\'{E}S;
and during two visits of the first author at Ritsumeikan University (Kyoto, Japan), supported by grants of this University.
 We thank these institutions.
\end{ack}

\section{Main result and corollaries}\label{results_sec}

Let us first recall some classical notions, vocabulary
 and elementary
 facts, and give
some precisions.

 Let $G$ be a connected {real} Lie group. A 
   $G$-\emph{Lie foliation} $\F$ on a manifold $M$
 is defined by a maximal foliated atlas of
  local coordinate charts
 $$(\phi_i:U_i\to V_i\times W_i)_{i\in I},$$
 where $U_i$, $V_i$, $W_i$ are open subsets in $M$, $\R^p$,
 $G$, respectively; the coordinate changes being
 smooth ($C^\infty$) diffeomorphisms of the form
 $$\phi_j\circ\phi_i\mun:(x,g)\mapsto(f_{i,j}(x,g),\gamma_{i,j}g)$$
 with $\gamma_{i,j}\in G$.
 One calls $G$ the \emph{structural group}. Clearly, one
 can assume $G$ to be simply connected; however it is more
  practical
 for our sake not to make this assumption.
 
For any manifold $N$ and any smooth map $g:N\to M$ transverse to $\F$,
  the pullback $g^*(\F)$
   of $\F$ is also a $G$-Lie foliation on $N$. We call such a pullback \emph{faithful} if moreover the preimage of
   every leaf of $\F$ is connected, hence a single leaf of
   $g^*(\F)$.
   
 Such a foliation becomes a fibre bundle
 over $G$ when lifted to the universal
   cover of the ambient manifold; precisely, E. Fedida
   proved the following structure theorem.
 \begin{prop}[\cite{Fedida}] Let $M$ be a closed connected manifold; denote by $\widetilde{M}$
 its universal cover. A foliation $\F$ on 
  $M$ is a $G$-\emph{Lie foliation} if and only if
   one has
\begin{itemize}
\item A surjective bundle map
 $\dev : \wt{M} \to G$ whose fibres cover the leaves of $\F$;
\item A group homomorphism $\hol : \pi_{1} M\to G$ such that for every $\gamma\in\pi_1M$ and $x\in\widetilde M$~:\[\dev(\gamma x)=\hol(\gamma) \dev(x).\]
\end{itemize} 
\end{prop}
One calls $\dev$ the \emph{developing map;} 
$\hol$ the \emph{holonomy representation;}
and the image subgroup $\hol(\pi_1M)\subset G$,
 the \emph{holonomy group} of $\F$.

Clearly, for every $g\in G$ one has~:
$$\pi_0(\dev\mun(g))\cong\pi_1G.$$
Every connected component of the fibre $\dev\mun(g)$
 is the universal cover of a leaf of $\F$ (since $\pi_2G=0$).

The codimension of the foliation
 $\F$ is the dimension of the structural group.
 The pair $(G,Hol(\pi_1M))$ of the structural group
 together with the holonomy subgroup characterizes the holonomy pseudogroup of $\F$ up to local isomorphism in Haefliger's sense \cite{Haefliger}
  (also known as Haefliger equivalence);
 one can call this pair the \emph{transverse structure} of the Lie
 foliation $\F$. The transverse structure is invariant under every
 faithful pullback.
 
   Note that $\F$ is minimal (in the
  sense that every leaf is dense in $M$) 
  if and only if the holonomy group $\hol(\pi_1M)$ is dense in $G$. We will usually assume that $\F$ is minimal.
   For Lie foliations,
 this restriction is not a serious one: indeed, in
 the general case, the closure of any leaf is a
 compact submanifold in which the foliation is a minimal Lie foliation (see \cite[Proposition 4.2, p.\ 116]{Molino}).
 
The leaves of $\F$ have no holonomy in Reeb's sense.
For any Riemannian metric on $M$, the leaves are two by two bi-Lipschitz diffeomorphic; so, there is an intrinsic quasi-isometry type, in Gromov's sense, for the universal covers of the leaves of $\F$.

\begin{ex}[Homogeneous Lie foliations]\label{hom_ex} 
Consider a short exact sequence of  Lie groups
$$1\to H\to L\to G\to 1$$
($H$ and $L$ being not necessarily connected) together with
a uniform (cocompact) lattice $\Gamma\subset L$ and a compact subgroup $K\subset H$, such that $L/K$ is connected and that for every $\ell\in L$~:
 \begin{equation}\label{eq_mfd}
 \Gamma\cap \ell K\ell\mun=\{1\}
 \end{equation}
  (for example, $\Gamma \cap H$ is torsion-free).

  Then,
 the left action of $\G$ on $L/K$ is free and properly discontinuous.
 Consider the symmetric space $X:=H/K$.
 The product foliation on $L/K$
  parallel to  $X$ being $\Gamma$-left invariant,
 one gets, on the locally homogeneous space $M:=\Gamma\backslash L/K$,
  a $G$-Lie foliation $\F$, called \emph{homogeneous}.
  We call $L$ the
  \emph{total group.} Actually,
  in this paper we shall only meet
   the special case where  $L$ is the direct product $H\times G$
   (\emph{direct product total group}).

   Clearly:
  \begin{itemize} 
   \item The leaves of $\F$ are locally isometric to $X$;
\item The holonomy group of $\F$
 is the projection of $\Gamma$ in $G$;
 \item If $X$ is 
 connected and simply connected,
  then the fundamental group of the leaves is $\Gamma\cap H$;
  \item If there exists any
   homogeneous $G$-Lie foliation with direct product total group, then $G$ must be unimodular.
  \end{itemize}
\end{ex}

One denotes by $\bH^d_{\RR}$ (resp.\  $\bH^d_{\CC}$) the real (resp.\ complex)
hyperbolic space of real (resp.\ complex) dimension $d$. A diffeomorphism between two
Riemannian manifolds is called a \emph{homothety} if the pullback of the target metric
is a constant times the source metric. Our main result is the following. Recall that $G$ denotes an arbitrary connected real Lie group.

\begin{thm} \label{main1_thm}
Let $X$ be a product of nonflat irreducible Riemannian symmetric spaces of non-compact types,
without factor homothetic to $\bH^2_{\RR}$.

 Let $(M,\F)$ be a connected closed manifold with a minimal $G$-Lie foliation,
bearing a $C^{0}$ Riemannian metric
  whose restriction to every leaf is smooth and locally isometric to $X$.

Then,  $(M,\F)$ {is finitely covered} by
a homogeneous $G$-Lie foliation whose total group is $\Isom(X)\times G$.
\end{thm}
{ Theorem \ref{main1_thm}
will be proved in Sections \ref{hyp_sec} to \ref{gen_sec},
case by case.}
\begin{rms}\label{no_cover_rms}
If moreover $X$ is irreducible, or more generally if $X$ does not
admit
 two homothetic irreducible factors,
then there is no use to pass to a finite cover:
$(M,\F)$ is smoothly conjugate to
a homogeneous $G$-Lie foliation whose total group is $\Isom(X)\times G$.
\end{rms}

\begin{rms}
 Of course, the covering
in Theorem \ref{main1_thm} is understood to be smooth;
however (as the proof will suggest), as 
a price to pay for this smoothness, the covering
will in general
not be locally isometric with respect to
the two leafwise locally symmetric metrics. In the same way,
in Remark \ref{no_cover_rms}, the smooth conjugation
will in general
not be isometric with respect to
the two leafwise locally symmetric metrics. See also
Pansu-Zimmer \cite{PZ}.
\end{rms}
The following quasi-isometric version will
follow by using theorems of Kleiner-Leeb \cite{KL} and Pansu \cite{Pansu} (Section \ref{qi_sec}).

\begin{thm} \label{main2_thm}
Let $X$ be a product of nonflat irreducible Riemannian symmetric spaces of non-compact types,
without factor homothetic to $\bH^d_{\RR}$  ($d\ge 1$)
 nor $\bH^d_{\CC}$ ($d\ge 1$).
 
 Let $(M,\F)$ be a connected closed manifold with a minimal $G$-Lie foliation
such that the universal cover of every leaf
 is quasi-isometric to $X$.
 
{Then,} $(M,\F)$ is finitely covered by
 a faithful
 pullback of a homogeneous $G$-Lie foliation
  whose total group is $\Isom(X)\times G$.
\end{thm}

Neither Theorem \ref{main1_thm} nor Theorem \ref{main2_thm} would hold
 if $X$ were the Poincar\'{e} disk. Indeed, there exists (\cite[Theorem 1.5]{HMM})
 a $\PSL(2,\R)$-Lie foliation $\F$ whose leaves are the $2$-sphere $S^{2}$ minus a Cantor set,
 and whose holonomy group $\Gamma\subset\PSL(2,\R)$ is a cocompact irreducible lattice in
 $\PSL(2,\Q_2)\times\PSL(2,\R)$. Hence, for every short exact sequence
 of \emph{real} Lie groups  as in Example \ref{hom_ex} with $G=\PSL(2,\R)$,
 the group $\Gamma$ cannot be the image in $\PSL(2,\R)$
 of any cocompact lattice in
 $L$, by the Margulis superrigidity
 theorem --- or if one likes better, by an argument of cohomological dimension.
 The same holds for every finite-index subgroup in $\Gamma$.
 In other words, $\Gamma$ cannot be, even virtually, the holonomy group of any
 homogeneous $\PSL(2,\R)$-Lie foliation.
 On the other hand,
 after Candel's uniformisation theorem \cite{Candel}, $\F$ admits a
  $C^{0}$ Riemannian metric
   whose restriction to every leaf is smooth of curvature $-1$.

\medbreak

Before stating some more corollaries of Theorem~\ref{main1_thm}, 
let us first recall a second important family of
 Lie foliations.

\begin{ex}[Suspension Lie foliations]\label{susp_ex}
Given a \emph{compact} connected Lie group $G$, a closed connected manifold $V$ and a representation $\rho:\pi_1V\to G$, its suspension $\F$ on $M:=V\times_\rho G$ is a $G$-Lie foliation, which is minimal if $\rho(\pi_1V)$ is dense in $G$. If moreover $V$ is a locally homogeneous manifold, then $\F$ is also a homogeneous Lie foliation. 
\end{ex}

Once we know that the given foliation is
homogeneous up to a finite covering,
we can apply the theory of lattices to study the global structure.
Here are some consequences that generalise the results of Zimmer.
The first one is a refinement of Theorem \ref{main1_thm} for the case where $X$ is irreducible:
\begin{cor}\label{dic_cor}
Let $X$ be an \emph{irreducible} symmetric space of non-compact type of
real dimension $n\ge 3$.
 Let $(M,\F)$ be a minimal Lie foliation on a closed manifold. Assume that $\F$ 
bears a $C^{0}$ metric whose restriction to every leaf is
smooth and locally isometric to $X$. Then, $(M,\F)$
is finitely covered by either a homogeneous Lie foliation
(Example \ref{hom_ex}) with irreducible $\G$,
or a suspention Lie foliation (Example \ref{susp_ex})
whose fibre is  a compact quotient group of $G$.
\end{cor}

When the model symmetric space $X$ is reducible, the foliation can be somehow a composite
of the homogeneous case with the suspension case:
 a suspension foliation whose fibres are homogeneous Lie foliation given by an irreducible lattice (see Proposition \ref{lochom_prop} below).
\medbreak
By Theorem~\ref{main1_thm} {and the
Molino structure theory for Riemannian foliations}, we obtain a rigidity result for minimal Riemannian foliations with locally symmetric leaves (see Theorem \ref{uni_thm}).

The next corollary
is a generalisation of Zimmer's arithmeticity theorem of holonomy groups (\cite[Theorem~A-(3)]{Zimmer}).

\begin{cor}\label{z_cor}
Let $X, M, \F$ be as in Theorem~\ref{main1_thm}.
 
 Then, the adjoint {$\Ad_{G}(\hol(\pi_1M))$} of the holonomy group of $(M,\F)$ is arithmetic in $\Ad_{G}(G)$. 
\end{cor}

 Zimmer proved this statement in the case where the leaves of $\F$ are simply connected and
where each irreducible factor of $X$ is of real rank at least {2}. He mentioned that his arithmeticity theorem may also
hold true for the cases
where the leaves are not simply connected. Our Corollary~\ref{z_cor} covers these cases. Quiroga-Barranco \cite{Quiroga} proved the arithmeticity for Lie foliations with totally geodesic leaves.

Finally, our Theorem \ref{main1_thm} together with Johnson's theorem~\cite{Johnson}
 give the following generalisation of \cite[Theorems A-(5) and B]{Zimmer}. 

\begin{cor}\label{codim_cor}
Let $(M,\F)$ and $X$ be as in Theorem~\ref{main1_thm}. Let $X=\prod_{i} X_{i}$ be the decomposition of $X$ into irreducible symmetric spaces. Unless  $(M,\F)$ is finitely covered by a suspension Lie foliation,
we have $\codim (\F) \geq d(X)$, where \[d(X)= \min_{i} \dim(\Isom(X_{i})).\]
\end{cor}

{Note that suspension foliations should be excluded from the original result \cite[Theorem A-(5)]{Zimmer} (see Example \ref{ex:Z}).}

{The structure of the article is as follows. In Section \ref{qa_sec}, we introduce a quasi-action of the structural group on the universal cover of a leaf, which is the basic tool in this paper. Sections \ref{hyp_sec} to \ref{gen_sec} are devoted to the proof of Theorem \ref{main1_thm}. In Section \ref{qi_sec}, we prove Theorem \ref{main2_thm}. The subsequent two sections present structure theorems for homogeneous Lie foliations, which are used to derive corollaries from Theorem \ref{main1_thm} and to apply the results to Riemannian foliations. We state several problems in the last section.}

\section{Quasi-action of the structural group}\label{qa_sec}
Here, one quickly recalls some classical material about Lie foliations, without any hypothesis
 on the geometry of the leaves; one fixes some notations; and one introduces what
  can be called
 a right \emph{quasi-action} of the structural group $G$ on the universal cover of the
 foliated manifold,
 lifting the action of $G$ on itself by right translations.

 We are given a connected Lie group $G$, a closed connected manifold $M$, and on $M$ a minimal $G$-Lie foliation $\F$.

Fix on $G$ a left-invariant Riemannian metric $\lambda$. Fix on $M$ a smooth Riemannian metric $\rho$ which coincides with $\lambda$ on the plane field $\rho$-orthogonal to $\F$ in $M$ (\emph{bundle-like metric}, after \cite{Reinhart}).
Endow $\widetilde M$ with the lift $\tilde\rho$ of $\rho$. For
 every $h\in G$, we denote by $\tilde\rho_h$ the \emph{intrinsic Riemannian}
 metric induced by $\tilde\rho$ on the fibre $\dev\mun(h)$.
   At every $x\in\widetilde M$, the differential of $\dev$ restricted to the subspace $\tilde\rho$-orthogonal to the fibre through $x$, is a linear isometry from this subspace onto $\tau_{\dev(x)}G$ (\emph{Riemannian submersion}).

For every $g\in G$, fix any shortest $\lambda$-geodesic segment $[e_G,g]$ from the neutral element $e_G$ to $g$ in $G$. For every $x\in\widetilde M$, the left-translated geodesic segment $\dev(x)[e_G,g]$ lifts uniquely through $\dev$ as a path $[x,y]$ in $\widetilde M$ which starts at $x$ and is 
 $\tilde\rho$-orthogonal to the fibres. Define $$\Phi(g)(x)=y
 \in\dev\mun(\dev(x)g)$$ as the endpoint
 of the lifted path.

In this paper, given a metric $s$ on a space $S$, we write $\vert x-y\vert_s$ for the $s$-distance
between two points $x, y\in S$; and given two maps $f, g: S'\to S$, we write $\Vert f-g\Vert_s$
for the supremum of the distances  $\vert f(x)-g(x)\vert_s$, where $x\in S'$.

\begin{lem} \label{dev_pts} For every $g, g', h\in G$ and $x\in\dev\mun(h)$ and  $\gamma\in\pi_1M$:
\begin{enumerate}
\item  $\vert\Phi(g)(x)-x\vert_{\tilde\rho}=\vert g-e_G\vert_\lambda$;
\item 
$\Phi(g)(\gamma\cdot x)=\gamma\cdot\Phi(g)(x)$;
\item
$\Phi(g)$ is a bi-Lipschitz self-diffeomorphism of $\widetilde M$
 with respect to the metric $\tilde\rho$; and
the Lipschitz ratios of $\Phi(g)$ and of $\Phi(g)\mun$ on $\widetilde M$
 with respect to $\tilde\rho$, seen as two real functions
of $g\in G$, are bounded on every compact subset of $G$;
\item $\Phi(g)\to\id$, \emph{uniformly on $\widetilde{M}$
 in the $C^1$ topology,}  as $g\to e_G$ in $G$;
 \item
 $\vert\Phi(gh)(x)-\Phi(h)(\Phi(g)(x)) \vert_{\tilde\rho_{\dev(x)gh}}$
  has a bound depending on $g, h\in G$ but not
   on $x\in\widetilde M$.
\end{enumerate}
\end{lem}
\begin{proof} (i) and (ii) are obvious by construction.
(iii) and (iv)  follow from (ii),
from the cocompacity of the action of $\pi_1M$ on $\widetilde M$,
and from the elementary properties of the geodesics of the complete Riemannian metric
$\lambda$. As for (v): 
 the considered distance is continuous with respect to $x\in\widetilde M$ (after (iii))
  and $\pi_1M$-invariant (after (ii)), hence bounded.
\end{proof}

\begin{rms}
The diffeomorphism $\Phi(g)$ depends on the choice of shortest geodesic segments from $e_G$ to $g$, but this choice will be irrelevant in the sequel.
\end{rms}
\begin{rms}
In general, the distribution $\tilde\rho$-orthogonal to the fibres being not integrable, $\Phi(gg')$ is \emph{not} the composite {$\Phi(g')\circ\Phi(g)$,} even for $g,g'$ close to $e_G$: for the moment, we have no genuine right $G$-action on $\widetilde M$.
\end{rms}
\begin{rms}
There is an alternative construction of the quasi-action $\Phi$, which would also work for our proof. Every $g\in G$ decomposes as $\exp(v_1) \cdot \cdots \cdot \exp(v_m)$, where each $v_i$ is some left-invariant vector field. Let $\tilde v_i$ be the lift of $v_i$ to $\widetilde M$ which is orthogonal to the fibres of $\dev$. One defines $\Phi(g):=\tilde v_m^1\circ \cdots \circ \tilde v_1^1$, where $\tilde v_i^1$ is the time one map of the flow generated by $\tilde v_i$. 
\end{rms}

It is convenient, for every $h, k\in G$, to introduce the restriction
  \[\Phi_h^{k}=\Phi(h\mun k)\vert_{ \dev\mun(h)} : \dev\mun(h) \to \dev\mun(k)\]
which is a bi-Lipschitz diffeomorphism between the two fibres with respect to the
metrics $\tilde\rho_h$, $\tilde\rho_k$.

\section{Rigidity of Lie foliations with hyperbolic leaves}\label{hyp_sec}
{Let us prove Theorem~\ref{main1_thm} in the case of real rank one.}
 Let $X$ be an irreducible symmetric space of non-compact type of real rank one, and of real dimension $n\ge 3$.
 In other words, by Cartan's classification, $X$ is one of the following {hyperbolic spaces}: $\bH_{\RR}^{n}$ (with $n\ge 3$), $\bH_{\CC}^{d}$ (with $d\ge 2$), $\bH_{\HH}^{d}$ (with $d\ge 1$) or $\bH_{\OO}^{2}$.
 
 One is given a connected Lie group $G$ which
 one may assume without loss of generality to be simply connected;
 $M$, $\F$, $\lambda$ as in Section \ref{qa_sec}; and
 one assumes that $\F$ bears a
 $C^{0}$ metric $\sigma$
  whose restriction to every leaf is
  smooth and locally isometric to $X$.
   Put $H := \Isom X$. Let $\eta$
   be the volume of the metric $\lambda$,
   hence a left-invariant Haar measure on $G$.
   
    We shall show that $\pi_1M$ embeds as a uniform lattice $\Gamma$ in $H\times G$, and that $\F$ is smoothly conjugate to the homogeneous foliation on $M_0:=\Gamma\backslash(X\times G)$ whose leaves are parallel to $X$.
\medbreak

Consider on every fibre $\dev\mun(h)$
 the symmetric Riemannian metric $\tilde\sigma_h$ which is the lift of $\sigma$.
 Since $\sigma$ is globally continuous,
  $\tilde\sigma_h$ and $\tilde\rho_h$ are
 equivalent, and the ratios of their norms have positive bounds not depending
  on $h$.

 { Both homotopy groups $\pi_1G$ and $\pi_2G$ 
 being trivial, $\dev\mun(h)$ is simply connected, hence isometric with $X$ for $\tilde\sigma_h$.
As such, this fibre has an ideal boundary at infinity $\partial\dev\mun(h)$,
 homeomorphic to the $(n-1)$-sphere.
 Let $\partial\widetilde{\mathcal{F}}$ be the disjoint union of the
boundaries $\partial\dev\mun(h)$, for $h\in G$.
Every element $\gamma\in\pi_1M$,
mapping isometrically $\dev\mun(h)$ onto $$\dev\mun(\hol(\gamma)h)$$
therefore extends
to a homeomorphism of $\partial\dev\mun(h)$ onto
 $$\partial\dev\mun(\hol(\gamma)h)$$
One thus gets a global
left action of $\pi_1M$ on the set $\partial\widetilde{\mathcal{F}}$.}

For every $h, k\in G$, the two metrics
  $\tilde\sigma_h$ and $\tilde\rho_h$ being equivalent on
  the fibre $\dev\mun(h)$,
  as well as   $\tilde\sigma_k$ and $\tilde\rho_k$ on
  the fibre $\dev\mun(k)$,
after Lemma \ref{dev_pts} (iii), the
 diffeomorphism
 $\Phi_h^{k}$ is bi-Lipschitz  with respect to $\tilde\sigma_h$ and $\tilde\sigma_k$,
 hence extends to a homeomorphism
  between the ideal boundaries at infinity:
 \[\partial\Phi_h^{k} :\partial\dev\mun(h) \to\partial\dev\mun(k).\]
 Globally, define for $g\in G$ the bijection $\partial\Phi(g)$
 of $\partial\widetilde{\mathcal{F}}$ such that for every $h\in G$ and $\xi\in\partial\dev\mun(h)$:
  \[\partial\Phi(g)(\xi)=\partial\Phi_h^{hg}(\xi).\]

 \begin{lem}\label{action_lem}
 For every $g, g'\in G$,  $\gamma\in\pi_1M$, $\xi\in\partial\widetilde{\mathcal{F}}$:
\begin{enumerate}
 \item
$\partial\Phi(gg')=\partial\Phi(g')\circ\partial\Phi(g)$.
\item $\gamma\cdot\partial\Phi(g)(\xi)=\partial\Phi(g)(\gamma\cdot\xi)$.
\end{enumerate}
\end{lem}
This follows at once from Lemma \ref{dev_pts} (ii) and (v).

  {One thus has a genuine
  set-theoretic right $G$-action on $\partial\widetilde{\mathcal{F}}$,
commuting with the left $\pi_1M$-action.}
Informally, thinking of $\partial\widetilde\F$ as a boundary for $\widetilde M$,
 one now seeks for an extension of this $G$-action into a $G$-action on the interior $\widetilde M$.
 
 For every $g\in G$, define $\Omega_g\subset G$ as the set of the $h\in G$
 such that there exists a $\tilde\sigma$-\emph{isometry}
 \[\Psi_h^{hg} : \dev\mun(h) \to \dev\mun(hg)\] inducing $\partial\Phi_h^{hg}$
  on the ideal boundaries at infinity. Notice that if such an isometry exists, then it is unique.
  Put $$\Omega:=\{(x,g)\in\widetilde M\times G\ /\ \dev(x)\in\Omega_g\}$$
  We get global mappings
  \[\Psi(g):\dev\mun(\Omega_g)\to\widetilde M:x\mapsto\Psi_{\dev(x)}^{\dev(x)g}(x),\]
  \[\Psi:\Omega\to\widetilde M:(x,g)\mapsto\Psi_{\dev(x)}^{\dev(x)g}(x).\]

After Lemma \ref{action_lem}, obviously:
  \begin{lem}\label{psi_lem}
  For every $g, g'\in G$ and $\gamma\in\pi_1M$,
  \begin{enumerate} 

   \item $\Omega_{g\mun}=\Omega_gg$;
   \item $\Psi(g\mun)=\Psi(g)\mun$ on $\dev\mun(\Omega_{g\mun})$;
   \item $\Omega_g\cap\Omega_{g'}g\mun\subset\Omega_{gg'}$;
   \item $\Psi(gg')=\Psi(g')\circ\Psi(g)$ on $\dev\mun(\Omega_g\cap\Omega_{g'} g\mun)$;
   \item $\Omega_{{{\hol}(\gamma)}g}={\hol(\gamma)}\Omega_g$;
   \item $\Psi(\hol(\gamma) g)=\gamma\circ\Psi(g)$ on $\dev\mun(\Omega_{\hol(\gamma)g})$.
  \end{enumerate}
\end{lem}

\begin{lem}
\label{continuity_lem}\

(i) The subset $\Omega$ is closed in $\widetilde M\times G$. 

(ii) The map $\Psi$ is continuous on $\Omega$.
\end{lem}

\begin{proof}
The proof essentially falls to Shchur's
 qualitative version of the Morse lemma  (\cite[Theorem 3]{Shchur}).
 It applies to every symmetric space $X$
 of negative curvature. After Shchur:
 \emph{if $\varphi : X \to X$ is an $(L,C)$-quasi-isometry and if
 $\psi : X \to X$ is an isometry at finite distance from $\varphi$, then we have
\begin{equation}\label{bd_eq}
\| \varphi(x) - \psi(x) \|_{X} < D,    
\end{equation}
where $D$ is a constant which depends only on $X$, $L$ and $C$.}

Consider a sequence $(x_i,g_i)\in\Omega$
converging in $\widetilde M\times G$ to $(x,g)\in\widetilde M\times G$. 
Put $h_i:=\dev(x_i)$ and $h:=\dev(x)$.

(i)  After Lemma \ref{dev_pts} (i) and (iii), there is a $L>0$
such that for every $i$, the diffeomorphism
 $\Phi_{h_i}^{h_ig_i}$
 between the fibres $\dev\mun(h_i)$ and $\dev\mun(h_ig_i)$
  is $L$-bi-lipschitz with respect to the metrics
   $\tilde\sigma_{h_i}$ and  $\tilde\sigma_{h_ig_i}$.
After the Morse-Shchur lemma,
$$\Vert\Phi_{h_i}^{h_ig_i}-\Psi_{h_i}^{h_ig_i}
\Vert_{\tilde\sigma_{h_ig_i}}$$
has a bound $D$ not depending on $i$. Then,
the continuity of $\tilde\sigma$ and the  Ascoli-Arzela theorem provide a subsequence
of $\Psi_{h_i}^{h_ig_i}$ converging to an isometry $$\Psi_h^{hg}:
\dev\mun(h)\to\dev\mun(hg)$$
at  distance at most $D$ from $\Phi_h^{hg}$; hence $(x,g)\in\Omega$.

(ii) It is enough to verify that if moreover
$(x,g)\in\Omega$, then after passing to some subsequence,
 $\Psi(x_i,g_i)$ converges to $\Psi(x,g)$.
Just as before, applying
 Lemma \ref{dev_pts} (i) and (iii),
 the Morse-Shchur lemma, and the Ascoli-Arzela theorem, one concludes that
 after passing to some subsequence,
$\Psi_{h_i}^{h_ig_i}$ converges to an isometry $$\psi:\dev\mun(h)\to\dev\mun(hg)$$
at finite distance from $\Phi_h^{hg}$. In particular, $\psi=\Psi_h^{hg}$; and thus
$\Psi_{h_i}^{h_ig_i}(x_i)$ converges to $\Psi_h^{hg}(x)$.
\end{proof}

\begin{lem}
\label{isometry_lem}\

(i) $\Omega=\widetilde M\times G$.
 
(ii) The Lie group $G$ is unimodular.
\end{lem}

\begin{proof}
One proves (i) and (ii) in the same time, applying
the barycentre method of Furstenberg,
Douady-Earle and Besson-Courtois-Gallot \cite{BCG} in our foliated situation. There are two steps: constructing the barycentre map,
and proving that this map is fibrewise isometric a.e. The hypothesis
that the dimension $n$ of $X$ is at least $3$
 will be essential here.

 Fix $g\in G$; consider the right translation
  $R_{g}$ mapping every $h\in G$ to $hg\in G$,
   and the positive real constant $c(g)$ such that  $(R_{g})_{*}\eta = c(g)\eta$.   
We shall prove that $c(g)=1$ and that $\Omega_g$ is of full measure in $G$,
which  because of Lemma \ref{continuity_lem} (i) implies that $\Omega_g=G$.
After Lemma \ref{psi_lem} (i),
 changing if necessary $g$ to $g\mun$, one can assume that $c(g)\le 1$.
   
\medbreak
\emph{{Step 1.}\ Constructing the barycentre map ---}
 For every $h\in G$, every $x\in\dev\mun(h)$ defines a visual probability
 Borelian measure $\mu_x$ on $\partial\dev\mun(h)$, and thus an image Borel measure \[(\partial\Phi_h^{hg})_*(\mu_x)\] on $\partial\dev\mun(hg)$, which has no atoms
 since $\partial\Phi_h^{hg}$ is a homeomorphism between the boundaries;
 one defines \[\Psi_h^{hg}(x)\in\dev\mun(hg)\] as the Furstenberg
  barycentre of this image measure. 
  One  thus gets a globally Borelian, fibrewise 
   barycentre map \[\Psi(g):\widetilde M\to\widetilde M\]
which after \cite[Sections 3 through 5]{BCG}, and since $n\ge 3$,
 is of class $C^1$ \emph{in every fibre}
 $\dev\mun(h)$; and $\vert \operatorname{Jac}(\Psi_h^{hg})\vert\le 1$ in the fibre
  with respect to $\tilde\sigma_h$ and to $\tilde\sigma_{hg}$. 
\medbreak
\emph{{Step 2.}\ The barycentre map is fibrewise isometric ---}
Consider on every fibre $\dev\mun(h)$ the $\tilde\sigma_h$-volume measure $v_h$,
and on $\widetilde{M}$ the positive Borel measure
\[\tilde\mu:=\int_G v_hd\eta(h).\]
One has
the disintegration formula
\[\Psi(g)_*(\tilde\mu)=c(g)\int_G\vert \operatorname{Jac}(\Psi_h^{hg})\vert v_hd\eta(h).\]

On the other hand, the Haar measure $\eta$ being left-invariant and the fibrewise
metric $\tilde\sigma$ being $\pi_1M$-invariant, $\tilde\mu$ is $\pi_1M$-invariant.
 After Lemma \ref{psi_lem} (vii) and (viii), $\Psi(g)$ passes to the quotient,
 and one gets a Borelian mapping \[\psi(g) : M \to M,\] which is of class $C^1$
 in every leaf and satisfies $\vert \operatorname{Jac}(\psi(g))\vert\le 1$
 with respect to $\sigma$ in every leaf.

In case the vector bundle $T\F$ over $M$ is not orientable,
consider over $M$ the $2$-sheeted cover $M'$ orienting the pull back of $T\F$;
otherwise, put $M'=M$.
We restrict $\tilde\mu$ to any Borelian fundamental domain for the action of $\pi_1M'$
on $\widetilde{M}$, and we project the restricted measure into $M'$. We thus get on $M'$
a positive Borel measure $\mu$ such that
  \[\int_{M'}\mu=\int_{M'}\psi(g)_{*}(\mu)=c(g)\int_{M'}\vert \operatorname{Jac}(\psi(g))\vert d\mu.\]

Since $\mu$ is finite, $c(g)=1$, and the identity
 $\vert \operatorname{Jac}(\psi(g))\vert=1$ hold
  $\mu$-almost everywhere on $M$. By Fubini's theorem, there is a subset of full measure $\Omega_g\subset G$
 such that for every $h\in\Omega_g$, the identity
   $\vert \operatorname{Jac}(\Psi_h^{hg})\vert=1$
 holds $\tilde\sigma$-almost everywhere on $\dev\mun(h)$, hence everywhere on $\dev\mun(h)$ by continuity.
   After \cite[Proposition 5.2]{BCG}, 
  since $n\ge 3$, this equality implies that $\Psi_h^{hg}$ is actually an isometry
  between the two fibres $\dev\mun(h)$ and $\dev\mun(hg)$; and one has $\partial\Psi_h^{hg}=\partial\Phi_h^{hg}$.
 \end{proof}

 After Lemmas \ref{psi_lem} and \ref{continuity_lem}, one gets a continuous right action $$(x,g)\mapsto\Psi(g)(x)$$ of $G$ on $\widetilde{M}$, lifting the right action of
 $G$ on itself,
  commuting with the left action of
 $\pi_1M$ on $\widetilde{M}$, and $\tilde\sigma$-isometric in the fibres of $\dev$.
  One also writes for short $x*g$ instead
 of $\Psi(g)(x)$.

Such a right $G$-action can alternatively be regarded as
 a ($C^0$) conjugation between $(M,\F)$ and a homogeneous Lie foliation $(M_0,\F_0)$. Namely, identify isometrically the fibre $\dev\mun(e_G)$ with $X$.
  One has a $C^0$ global trivialisation of $\widetilde M$ as an $X$-bundle over $G$:
\[
\tilde c : \widetilde M\to X\times G ; \,\, x\mapsto(x*\dev(x)\mun,\dev(x)).
\]
Also, the left $\pi_1M$-action and the right $G$-action on $\widetilde M$
 provide a representation 
$$ r : \pi_1M \to\Isom(X) $$
$$ r(\gamma):X \to X:x \mapsto(\gamma\cdot x)*\dev(\gamma\cdot x)\mun$$     
Put for short $H:=\Isom(X)$; consider the product representation
 \[r\times\hol:\pi_1M\to H\times G.\]
 Let 
 \begin{align*}
 \Gamma & :=(r\times\hol)(\pi_1M)\subset H\times G, \\
 M_0 & :=\Gamma\backslash(X\times G).
 \end{align*}
 Let $\F_0$ be the foliation of $M_0$ parallel to $X$, and let $c:M\to M_0$
 be the quotient of $\tilde c$.
 Since $\pi_1M$ acts freely on $\widetilde{M}$ and  $\tilde c$ is a $\pi_1M$-equivariant
 homeomorphism, clearly $r\times\hol$ is one-to-one and $c$ is a homeomorphism.
  In particular,
  $M_{0}$ is compact, $\Gamma$ is a uniform lattice in $H\times G$,
   isomorphic to $\pi_1M$, and  $\F_0$ is a homogeneous $G$-Lie foliation whose leaves are the 
   orbits of the $H$-action.  By construction, the leaves of $\F$ are the preimages of the leaves
 of $\F_0$ through $c$. Since $\G$ acts freely on $X \times G$, the condition $\G \cap hKh^{-1} = \{1\}$ in \eqref{eq_mfd} holds. 

{The map $c$ conjugates $\F$ to $\F_0$,
and its restriction to each leaf of $\F$ is an isometry.
However, $c$ is not globally smooth in general.
In the next lemma, we perturb $c$ to get a smooth conjugation which may not be isometric in the leaves.}

\begin{lem}\label{smooth_lem} $\F$ is \emph{smoothly} conjugate to $\F_0$.
\end{lem}

\begin{proof} One uses the notation $\nb(x)$
meaning an arbitrary open neighborhood of $x$.
 The conjugating homeomorphism $c:M\to M_0$
is not only continuous. Indeed, put $n:=\dim(\F)$ and $q:=\codim(\F)$.
 Recall that a smooth local coordinates system
$x_1, \dots, x_n$, $y_1, \dots, y_q$ in $M$, abbreviated as $(x,y)$,
is called $\F$-\emph{distinguished} if $\partial/\partial x_1$,
\dots, $\partial/\partial x_n$ span $T\F$. The conjugation means that 
 for every $x\in M$ and every two distinguished smooth local coordinates systems
 $(x,y)$ on $\nb(x)\subset M$ and $(x',y')$ on $\nb(c(x))\subset M_0$,
  on $\nb(x)$ the map $c$
 has the form
 \begin{equation}\label{conjugation_eqn}
 (x,y)\mapsto(f(x,y),g(y)).
 \end{equation}

We claim:
\begin{enumerate}[label=(\alph*)]
\item $c$ is transversely $C^\infty$, which means that $g$ is $C^\infty$;
\item $c$ is globally $C^{\infty,0}$, which means
  that for
every integers $r\ge 0$ and $$1\le j_1,\dots, j_r\le n$$
the leafwise partial derivate $$\partial^rf/\partial x_{j_1}\dots\partial x_{j_r}$$
 exists and
is continuous.
\end{enumerate}

For, (a) follows from the very construction of $c$;
 while (b) follows at once from the following facts:
  $c$ is continuous on $M$; and $c$ is
 smooth and
 isometric in the leaves;
and on the space of
  the isometries of $X$, the
  $C^0$ topology coincides with the $C^\infty$ topology.
  
  It is known that such a conjugation between two $C^\infty$ foliations is
  $C^0$-closely approximable by a $C^\infty$ conjugation. However,
  having not found any
  reference in the litterature, we give a proof.
   Indeed, we shall
  smoothen the conjugation by composition with
  a $C^0$ isotopy in $M_0$, tangential to $\F_0$ and $C^0$-close to the identity.

 After W.\ Thurston, $M$ admits an arbitrarily fine, smooth triangulation $T$ in
  \emph{general position} with respect to $\F$ (see \cite{Thurston}).
   We choose $T$ so fine that for every
  cell $\alpha$ of $T$,
  the image $c(\alpha)$ lies in the domain $U'$ of an $\F_0$-distinguished
 coordinates system $(x',y')$. Let $d:=\dim(\alpha)$.
  By induction, we can assume that $c$
   already smoothly embeds some open neighborhood $V$ of
    the $(d-1)$-skeleton of $T$, into $M_0$.
    We shall modify $c$ on a small neighborhood of $\alpha$ relatively
    to a small neighborhood of $\partial\alpha$, and obtain
  a conjugation $\tilde c$ still transversely $C^\infty$
  and globally $C^{\infty,0}$, moreover smooth on $\nb(\alpha)$.
    Write $\I$ for the compact interval
   $[-1,+1]$.
   The general position
   implies
  that one has a compact neighborhood $N\cong\I^{n+q}$ of $\alpha$ in $M$
  such that
  
  \begin{enumerate}
  \item On $N\cong\I^n\times\I^{q}$, the foliation $\F$
  is the slice foliation parallel to $\I^n\times 0$;
  \item On $N\cong\I^n\times\I^{q}$, the map
  $c$ has the form
  (\ref{conjugation_eqn})
  with respect to the standard coordinates $(x,y)$
   on $\I^n\times\I^{q}$ and to the coordinates $(x',y')$ on $U'$.
  \end{enumerate}
  
   Then, the regularities (a) and (b) hold on $N$; the map $g$ is a
  smooth embedding of of $\I^q$ into $\R^q$; and for every $y\in\I^q$, the map
  $f_y:x\mapsto f(x,y)$ is a smooth embedding of $\I^n$ into $\R^n$. Since $f$
  is already smooth on $\nb(\partial\alpha)$, after changing the 
  smooth coordinates system $(x',y')$ on $\nb(c(\alpha))$, we can moreover arrange that
  $\partial f/\partial y=0$ on $\nb(\partial\alpha)$. To fix ideas, we
  can also make
  $g(y)=y$ on $\I^q$.
  
Fix a compact neighborhood $K$ of $\alpha$ interior to $N$.
  On a small neighborhood of $K$ in $N$, following a standard smoothing method,
   we approximate $f$ by a sequence of
  smooth functions $f_k:=f*_y\rho_k$ obtained as the convolution product of $f$,
  with respect to the variable $y$, with a sequence
  of smooth nonnegative bump functions $y\mapsto\rho_k(y)$ whose support is the ball
  of centre $0$ and radius $1/k$ in $\R^q$ and whose integral is $1$.
  Note that for $k$ large enough, one has
   $f_k=f$ on $\nb(\partial\alpha)$.
  Since $f$ is continuous, one has $f_k\to f$ uniformly on $K$.
  By (b), one has $\partial f_k/\partial x\to\partial f/\partial x$ uniformly on $K$.
  Hence, for $k$ large enough, and for every $y\in\I^q$, the map
  $f_k$ induces a smooth embedding of
  $\operatorname{Int}(K)\cap(\I^n\times y)$ into $\R^n$.
  Let $(\phi,1-\phi)$ be a smooth partition of the unity on $N$ subordinate to
  the open cover $(\operatorname{Int}(K),N\setminus\alpha)$. For $k$ large enough,
  $\tilde f_k:=(1-\phi)f+\phi f_k$ induces, for every $y\in\I^q$,
   a smooth embedding of $\I^n\times y$ into $\R^n$. Fix such an $n$.
  The map $\tilde c:M\to M_0$ coinciding with $c$ on $M\setminus K$ and with
  $(x,y)\mapsto(\tilde f_k(x,y),y)$ in $N$ works.
  \end{proof}

Theorem~\ref{main1_thm} is proved, for $X$ of real rank one.
{Note that in this case $(M,\F)$ itself is smoothly conjugate to a homogeneous foliation and we do not need to take a finite covering (See Remark \ref{fincov_rms} below).}

\begin{rms}
By a theorem of Pansu \cite{Pansu}, any quasi-isometry from $\bH_{\HH}^{d}$ to itself or $\bH_{\OO}^{d}$ to itself can be uniquely approximated by an isometry up to a bounded error. For the case where $X= \bH_{\HH}^{d}$ or $\bH_{\OO}^{d}$,
one could 
use this result to construct the right $G$-action, in place of barycentre maps.
\end{rms}

\section{Rigidity of Lie foliations with leaves locally isometric to an irreducible symmetric space of higher rank}\label{rank_2_sec}
{Let us prove Theorem~\ref{main1_thm} in the irreducible cases of rank $\ge 2$.}
Let as before $G$ be a simply
 connected Lie group, and $(M,\F)$ be a closed connected manifold with a minimal $G$-Lie foliation. Assume that $\F$ admits a $C^{0,\infty}$ metric whose restriction to every leaf is locally isometric to an irreducible symmetric space
 $X$ of non-compact type of real rank at least $2$. 
 In that case, the proof of Theorem~\ref{main1_thm} follows the same general lines
 as it did in the hyperbolic case, the main difference being that the construction
 of the right action $\Psi$ and the proof of its continuity
  now fall to {the Kleiner-Leeb rigidity theorem}
 instead of the barycentre method. Here are a few precisions; the other details
 are left to the reader.

\begin{lem}
\label{continuity2_lem} There is a {continuous}
 right action of the Lie group $G$
on the manifold $\wt{M}$, denoted $(x,g)\mapsto\Psi(g)(x)$,
commuting with the left action of $\pi_1(M)$, and such that
 $\Psi (g)$ sends isometrically $\dev\mun(h)$ onto $\dev\mun(hg)$,  for every $g, h\in G$.
\end{lem}

\begin{proof} 
Recall from Lemma \ref{dev_pts} (iii) that $\Phi_h^{hg}$ is a bi-Lipschitz diffeomorphism between the fibres $\dev\mun(h)$ and $\dev\mun(hg)$, whose
 Lipschitz ratio is uniformly bounded with respect to $g$ varying in any fixed compact subset of $G$. Both fibres being Riemannian symmetric space of non-compact type of real rank $\geq 2$, by \cite[Theorem~1.1.3]{KL} and the fact that any homothety of $X$ is an isometry (see, e.g., \cite{Ochiai}), there exists a unique isometry 
 \[
 \Psi_h^{hg} : \dev^{-1} (h) \to \dev^{-1} (hg)
 \]
 at finite distance from $\Phi_h^{hg}$; and this distance is bounded uniformly
  with respect to $g$ varying in any fixed compact subset of $G$. Define
  $\Psi (g)$ as the self-mapping of $\wt{M}$ whose restriction to each fibre
   $\dev^{-1} (h)$ is $\Psi_h^{hg}$.

Since, in restriction to each fibre, both $\Psi (h)\Psi (g)$ and $\Psi(gh)$ are isometries
 at finite distance from $\Phi (gh)$,
one has $\Psi(gh)=\Psi(h)\Psi(g)$
 by the uniqueness of such an isometry. For the same reason, $\Psi (g)\Psi (g^{-1})$ is the identity, and $\Psi(g)$ commutes with the left action of $\pi_1M$ on $\widetilde M$.

The continuity of $\Psi (g)(x)$ with respect to $(x,g)\in\wt{M}\times G$
 follows from the same argument as in the proof of Lemma \ref{continuity_lem},
 the differences being that $\Psi(g)$ is already defined all over $\wt{M}$,
  and that the Kleiner-Leeb theorem now assumes the role that the Morse-Schuhr lemma played
  in the hyperbolic case.
\end{proof}

The rest of the proof is the same as in the case where the leaves were
of real rank one, in the last section: 
one constructs a right $G$-action, a homogeneous Lie foliation $(M_{0},\F_{0})$ and a homeomorphism $(M,\F) \to (M_{0}, \F_{0})$ of class $C^{\infty,0}$, which can be approximated by a diffeomorphism $(M,\F) \to (M_{0}, \F_{0})$. Theorem \ref{main1_thm} is proved in the case
where the  leaves are irreducible and of real rank at least $2$.

\section{Rigidity of Lie foliations with locally symmetric leaves}\label{gen_sec}

{Let us prove Theorem \ref{main1_thm} in the general case.} 
Consider now
the general case: $$X = \prod_{i=1}^{\ell} X_{i},$$ where each $X_{i}$ is,
 for $i=1, \dots, \ell$, an irreducible Riemannian symmetric space of non-compact type, 
 of real dimension at least $3$. 

 The proof of Theorem \ref{main1_thm} in that general context will essentially 
 follow the lines of the
 proofs given above
 in the irreducible cases, but also asks for some specific arguments that we point out.
 \medbreak
Preliminaries --- It is convenient, to fix ideas,
to arrange that for every $1\le i, j\le\ell$,
either $X_i=X_j$
 or $X_i$ and $X_j$ are not isometric. After
 de Rham's unicity of the factorisation of $X$, its group of isometries splits as the
 semidirect product
 $$\Isom(X)\cong S_X\ltimes\prod_{i=1}^\ell\Isom(X_i),$$
 where $S_X\subset\mathfrak{S}_\ell$ is the finite group of the permutations $\tau$
 such that $X_i=X_{\tau(i)}$ for every $1\le i\le\ell$.
 Write $p_i:X\to X_i$ for the $i$-th projection.

 We can assume that $G$ is simply connected.
  By a \emph{fibrewise Riemannian bundle} over $G$, we mean a
  $C^0$ locally trivial fibre bundle $\pi:E\to G$ together with,
   on every fibre
  $\pi\mun(h)$,
  a smooth structure and a smooth Riemannian metric,
  globally continuous with respect to $h$.
  Clearly, the Whitney product of any two fibrewise
   Riemannian bundles
  over $G$ gives a fibrewise Riemannian bundle over $G$.
 
 Let $M$, $\F$, $\sigma$ be as in Theorem \ref{main1_thm}; 
 as in Section \ref{hyp_sec}, let $\tilde\sigma_h$
 be the lift of $\sigma$ in $\pi\mun(h)$.
 Hence, $(\widetilde M,\tilde\sigma)$ is a fibrewise Riemannian
 bundle over $G$.

\medbreak
 Step 1 --- Claim: there is a unique factorisation of
 $(\widetilde M,\tilde\sigma)$ as a Whitney product
  $$(\widetilde M,\tilde\sigma)=\prod_{i=1}^\ell(\widetilde M_i,\tilde\sigma_i),$$
 such that each factor $(\widetilde M_i,\tilde\sigma_i)$ is
 a fibrewise Riemannian bundle over $G$ whose fibres are
  isometric to $X_i$.
  
  Indeed, fix basepoints $x_i$ in $X_i$, and $x=(x_1,\dots,x_\ell)$ in $X$;
 write $\Isom(X_i,x_i)\subset \Isom(X_i)$ (resp.\ $\Isom(X,x)\subset \Isom(X)$) for the subgroup fixing $x_i$ (resp.\ $x$); put
  $$H_i:=\Isom(X_i,x_i)\times\prod_{j\neq i} \Isom(X_j),$$
 \begin{align*}
H:= & \prod_{i=1}^\ell \Isom(X_i), & 
K:= & \prod_{i=1}^\ell \Isom(X_i, x_i).
 \end{align*}
 
 The metric $\sigma$ being continuous,
 the bundle $\widetilde M$ is associated to a
 principal $\Isom(X)$-bundle of class $C^0$ over $G$.
 The base $G$ being simply connected, the corresponding principal $S_X$-bundle is trivial; hence,
  $\widetilde M$ is also associated to a principal $H$-bundle $P$ of class $C^0$ over $G$. 
  Since $H/K=X$ splits as the product of the factors $H/H_i=X_i$,
  it follows that $\widetilde M\cong P/K$ does split as the Whitney product, over the base $G$, of the fibrewise Riemannian bundles
  $\widetilde M_i:=P/H_i$ of fibre $X_i$, for $i=1,\dots,\ell$. The unicity of
  this factorisation of $\widetilde M$ is obvious from de Rham's unicity of the factorisation of $X$.

\medbreak
 Step 2 --- The action of $\pi_1M$ on $\widetilde M$ also somehow splits, as follows.
 
  Write  $q_{i}$ (resp.\ $\dev_i$) for the projection
  $\wt{M} \to \wt{M}_{i}$ (resp.\ $\widetilde M_i\to G$).
Since $\pi_{1}M$ acting on $\wt{M}$ permutes isometrically the fibres of $\dev$,
  in view of the unicity of the above factorization of $\widetilde M$,
   the action of $\pi_{1}M$ also
  permutes the factors $\wt{M}_{1}, \dots, \wt{M}_{\ell}$,
  in the sense that there are a group homomorphism $$\alpha : \pi_{1}M \to S_X$$ 
  and, for every $\gamma \in \pi_{1}M$ and each $1\le i\le\ell$,
  a homeomorphism $$\gamma_i:\widetilde M_i\to\widetilde M_{\alpha(\gamma)(i)}$$
  that maps isometrically every fibre $\dev_i\mun(h)$ onto the fibre
   $$\dev_{\alpha(\gamma)(i)}\mun(\gamma h)$$ and
  such that $$\gamma_i\circ q_i=q_{\alpha(\gamma)(i)}\circ\gamma$$
  Let $\Sigma=\ker \alpha$, a finite index subgroup in $\pi_{1}M$. 
   One thus obtains, on each factor $\widetilde M_i$, an action $(\gamma,x)\mapsto
   \gamma_i(x)$ of $\Sigma$
  which is obviously cocompact, in the sense that $\widetilde M_i$ contains
  a compact subset that meets every orbit. However,
  this action needs of course not be discrete for $\ell\ge 2$.

\medbreak Step 3 ---
The right quasi-action $\Phi$ of $\pi_1M$ on $\widetilde M$ also somehow splits, as follows.

 After Kleiner-Leeb \cite[Theorem~1.1.2]{KL} (see also Kap\-ov\-ich-Klein\-er-Leeb \cite{KKL}), for every quasi-isometry $\phi$ of $X$, there exist
a permutation $\per(\phi)\in \mathfrak{S}_\ell$,
and for each $i$ a quasi-isometry $\phi_{i}$
of $X_i$ into $X_{\per(\phi)(i)}$, such that $p_{\per(\phi)(i)} \circ \phi$ equals $\phi_{i} \circ p_{i}$ up to a bounded error. Clearly, the permutation $\per(\phi)$ is uniquely determined
by $\phi$; and two quasi-isometries of $X$ lying at finite maximal distance from each other
induce the same permutation.

 For every $h, k \in G$, recall from Lemma \ref{dev_pts}
 the bi-Lipschitz self-dif\-feom\-orphism
 $\Phi_h^k$ between the fibres $\dev\mun(h)$ and $\dev\mun(k)$.
  The Kleiner-Leeb theorem
 applied to  $\Phi_h^{k}$ provides a permutation $per(h,k)\in \mathfrak{S}_\ell$,
and for each $i$ a quasi-isometry $$\phi_i(h,k):\dev_i\mun(h)\to\dev_{per(h,k)(i)}\mun(k)$$ such that
 $q_{per(h,k)(i)} \circ \Phi_h^{k}$ equals $\phi_i(h,k)\circ q_{i}$ up to a bounded error.

 We claim that the permutation $per(h,k)$ is actually the identity for every $h, k\in G$.
 First, since for every $g, h, k\in G$
 the compose $\Phi_g^h\circ\Phi_h^k$ lies at finite distance from $\Phi_g^k$
 (Lemma \ref{dev_pts}, (v)), the uniqueness of the  Kleiner-Leeb permutation forces
 the identity $$\per(g,h)\circ\per(h,k)=\per(g,k)$$

 Then, $G$ being arcwise connected, it is enough to verify that every $h\in G$ has
 a neighborhood in $G$ in which $\per(h,\cdot)$ is the identity. By contradiction,
 if not, one would have in $G$ a sequence $(h_m)$ converging to $h$ for
 $m\to+\infty$, such that
 $\per(h,h_m)$ is a constant permutation $\tau\neq\id$. Fix a $1\le j\le\ell$ such that
 $\tau(j)\neq j$.
 On the other hand, since the Lipschitz ratios of the sequence $\Phi_h^{h_m}$
 are bounded (Lemma \ref{dev_pts}, (iii)), after Kleiner-Leeb the maximal distance between
 $\Phi_h^{h_m}$ and the product $\prod_i\phi_i(h,h_m)$ also has an upper bound $D$
 independant of $m$. In particular,
  in the fibre
 $\dev\mun(h_m)$ endowed with the metric $\tilde\sigma_{h_m}$,
  the image $\Phi_h^{h_m}(\dev_j\mun(h_m))$ lies in the $D$-neighborhood of the factor $\dev_{\tau(j)}\mun(h_m)$. Recall from
 Lemma \ref{dev_pts} (i) that for $m\to+\infty$, 
 the embeddings $\Phi_h^{h_m}$ of the fibre $\dev\mun(h)$ converge
 to the identity in $\widetilde M$. Hence,
  in the fibre $\dev\mun(h)$ endowed with the metric $\tilde\sigma_h$,
  the factor $\dev_j\mun(h)$ also lies in
 the $D$-neighborhood of the other factor $\dev_{\tau(j)}\mun(h)$, a contradiction.
 The claim is proved.

 For every $g\in G$ and each $1\le i\le\ell$, define $\Phi(g)_i:\widetilde M_i\to\tilde M_i$ as
 the collection of the $\phi_i(h,hg)$'s.

\medbreak
  Step 4 ---  For each
     $i$ such that the real rank of $X_i$ is $1$, 
     much as in
    the proof of Lemma \ref{isometry_lem}, following \cite{BCG},
     we consider the barycentre map $\Psi_{i}(g) : \wt{M}_{i} \to \wt{M}_{i}$ determined by $\Phi(g)_{i}$.  This $\Psi_{i}(g)$ does not expand the tangential volume form,
     according to \cite{BCG}. 
 \medbreak

Step 5 --- 
For each 
 $i$ such that the real rank of $X_i$ is at least $2$,
 much as in the proof of Lemma \ref{continuity2_lem},
 following
  \cite[Theorem~1.1.3]{KL}, there is a homeomorphism $\Psi_{i}(g)$ of $\wt{M}_{i}$
  which is
    at finite distance from $\Phi(g)_{i}$ over every compact subset of $G$,
    and which maps isometrically every fibre $\dev\mun(h)$ onto $\dev\mun(hg)$.
     Since $\Psi_{i}(g)$ is a fibrewise isometry, it preserves the tangential volume form of $\wt{M}_{i}$.
    \medbreak

 Step 6 ---   Let $\Psi(g) : \wt{M} \to \wt{M}$ be the Whitney product  over $G$
      of the maps $\Psi_{i}(g)$, for
     $1\le i\le\ell$. 
     {By the construction in the last step, $\Psi(g)$ does not expand the tangential volume of $\wt{\F}$.}
     Much as in Lemmas \ref{isometry_lem} and \ref{continuity_lem},
     but using the $\Sigma$-action on $\wt{M}$ instead of the $\pi_{1}M$-action, one shows
      that $G$ is unimodular, and that $\Psi(g)$ coincides,
       on almost every fibre of the developing map,
        with a fibrewise isometry defined on the whole $\wt{M}$.
        
\medbreak
         The end of the proof of Theorem \ref{main1_thm} is similar to the previous cases.

\begin{rms}\label{fincov_rms}
    {As the proof suggests, if $\Sigma = \pi_1M$, then $(M,\F)$ itself is smoothly conjugate to a homogeneous foliation: no finite
    covering is necessary. For example, this condition is satisfied if $X$ is irreducible, or more generally if no pair of factors of $X$ are mutually homothetic.}
\end{rms}

\section{{Quasi-isometric rigidity of Lie foliations}}\label{qi_sec}

{Let us prove Theorem \ref{main2_thm}.}
Assume that $X$ is a product of symmetric spaces of non-compact type which are neither $\bH^d_{\RR}$  ($d\ge 1$)
 nor $\bH^d_{\CC}$ ($d\ge 1$),
and that the universal covers of the leaves of $\F$,
or equivalently the fibres of $\dev$, are quasi-isometric to $X$. We apply theorems of Kleiner-Leeb and Pansu.

 First, fix quasi-isometries $\alpha : X \to\dev\mun(e_G)$ and $\beta :
 \dev\mun(e_G)\to X$ which are quasi-inverse to each other.
 Let $\operatorname{QI}(X)$ be the group of quasi-isometries of $X$. Define the map $$\Xi : \pi_{1}M \to \operatorname{QI}(X)$$ $$\Xi(\gamma)(x) =  \beta (\gamma^{-1} \cdot (\Phi(\gamma) \circ \alpha(x))),$$ which is a homomorphism up to a bounded error. By theorems of Kleiner-Leeb \cite{KL} and Pansu \cite{Pansu}, for each $\gamma \in \pi_{1}M$, there exists $\Psi(\gamma) \in \Isom (X)$ at a finite distance from $\Xi$. Thus,
  we have a homomorphism $$\Psi : \pi_{1}M \to \Isom (X).$$
  {As in the last section, up to a finite covering, we can assume that $\Psi$ maps each factor of $X$ onto itself.} Take a global section $\sigma : M \to X \times_{\Psi} M$, where $X \times_{\Psi} M$ is the quotient of $X \times \widetilde{M}$ by the $\pi_1M$-action given by $\gamma \cdot (x,z) = (\psi(\gamma)(x), \gamma \cdot z)$. Let $\tilde{\sigma} : \widetilde{M} \to X$ be its lift to $\wt{M}$. Now the image of $\Psi \times \hol : \pi_{1}M \to \Isom (X) \times G$ is a uniform lattice $\Gamma$. We can see that $\tilde{\sigma } \times \dev : \wt{M} \to X \times G$ induces a map $M \to  \Gamma \backslash (X \times G)$ such that $\F$ is the pull back of the homogeneous foliation on $\Gamma \backslash (X \times G)$.

\section{Homogeneous Lie foliations}

We prove the following structure result for homogeneous Lie foliations with locally symmetric leaves, which says that a finite covering of the foliation is obtained by certain suspension construction. The argument is similar to Zimmer's in \cite[Sections 5 and 10]{Zimmer}, although we use Auslander's theorem for uniform lattices instead of Zimmer's superrigidity for cocycles.
\begin{prop}\label{lochom_prop}
Let\begin{itemize}
\item $H$ be a semisimple Lie group with trivial centre, finitely many connected components and without nontrivial connected compact normal subgroup;
\item $K\subset H$ be a maximal compact subgroup;
\item $G$ be a simply-connected Lie group;
\item $\Gamma\subset H \times G$ be a uniform lattice such that $\Gamma\cap hKh\mun=\{1\}$ for every $h\in H$.
\end{itemize}
Consider,
as in Example \ref{hom_ex},
 the homogeneous $G$-foliation $\F$ on $\Gamma \backslash (H \times G) / K$. Assume that $\F$ is minimal, in other words, the projection of $\Gamma$ to $G$ is dense. Then, one of the following holds:
\begin{enumerate}
\item (Suspension whose fibres are homogeneous Lie foliations with irreducible lattices) $G$ is non-compact semisimple. $X$ is decomposed into a product of two symmetric spaces $X = X' \times X''$, and we have an irreducible uniform lattice $\G''$ of $\Isom(X'') \times G$ such that a finite covering of $(M,\F)$ is a flat fibre bundle over a locally symmetric space locally isometric to $X'$ whose fibres are homogeneous $G$-Lie foliations on $\G'' \backslash (X'' \times G)$ with simply-connected leaves.
\item (Example \ref{susp_ex}, suspension Lie foliations) There exist a uniform lattice $\G'$ of $H$ and a compact quotient group $G'$ of $G$, of the same dimension, such that a finite covering of $(M,\F)$ is diffeomorphic to the suspension $G'$-Lie foliation on a $G'$-bundle over a compact locally symmetric space $\G' \backslash H/K$. If moreover $H$ has no connected Lie subgroup locally isomorphic to neither of $\PSO(n,1)$ nor $\PU(n,1)$, then $G'$ is semisimple.
\end{enumerate}
\end{prop}

\begin{proof}
For any quotient group $N$ of $H \times G$, denote the canonical projection $H \times G \to N$ by $p_{N}$. The intersection $\Gamma \cap G$, being normal in $\Gamma$,
is also normal in $p_{G}(\Gamma)$, which is dense in $G$ by assumption;
hence, $\Gamma \cap G$ is normal in $G$; and hence central in $G$,
since $\Gamma \cap G$ is discrete in $G$. 

First, let us consider the case where $p_{H}(\Gamma)$ is discrete
in $H$, and show (ii). By Selberg's lemma \cite{Selberg}, $p_{H}(\Gamma)$ has a torsion-free finite index subgroup. Thus, by replacing $(M,\F)$ with a finite covering, we can assume that $p_{H}(\Gamma)$ is torsion-free. Then, we can regard $\G \backslash (H \times G)$ as the total space of a fibre bundle whose base is $V:=p_{H}(\Gamma) \backslash H/K$ and whose fibre is $G' := (\G \cap G) \backslash G$. This total space being compact,
 both $V$ and $G'$ are compact. Put $\G' := (\Gamma \cap G) \backslash \G$. It is immediate that $\F$ is diffeomorphic to the suspension foliation of the homomorphism $\pi_1V \to\G' \to G'$ over $V$.
 
 Let us show the latter sentence of (ii). Assume further that $H$ has no normal subgroup isomorphic to $\PSO(n,1)$ or $\PU(n,1)$. After~\cite[Assertion IX.6.18 (B)]{Margulis}, which follows from results of Deligne \cite{Deligne}, Raghunathan \cite{Raghunathan} and Margulis' arithmeticity theorem, $\Gamma/[\Gamma,\Gamma]$ is finite. So, the image of any homomorphism $\Gamma \to S^{1}$ is finite. If $G'$ had a nontrivial torus direct summand $T$, then $p_{T}(\G)$ could not be dense in $T$. Hence, $G'$ is semisimple.

Let us now consider the case where $p_{H}(\Gamma)$ is not discrete, and show (i).
 Let $F\subset H$ be the identity component of the closure
  $\overline{p_{H}(\Gamma)}$. Since $\overline{p_{H}(\Gamma)}$ is a Zariski-dense Lie subgroup of $H$, the algebra $\Lie(F)$ is invariant under the adjoint action of $H$;
in other words, $F$ is normal in $H$. Being moreover semisimple and centre-free, $H$ splits as the direct product of $F$ by a proper Lie subgroup $H'\subset H$. By construction, $p_{H'}(\Gamma)$ is discrete in $H'$.
Regard $\Gamma \backslash (H \times G)$ as the total space of a fibre bundle whose base is
$p_{H'}(\Gamma) \backslash H'$ and whose fibre is
 \[(\Gamma \cap(F\times G)) \backslash (F\times G)\]
The total space being compact,
the base and the fibre are compact.
Since moreover $p_{F}(\G)$ is dense in $F$,
the group $\Gamma \cap (F \times G)$ is an irreducible uniform lattice in $F \times G$.
 Also, by the same argument as the first paragraph of the present proof,
 $\G \cap F$ is central in $F$. Since $H$ is centre-free, $F$ is also
centre-free, hence $\Gamma \cap F= \{e\}$. 
Let us show that $G$ is semisimple. Let $R$ be the maximal normal solvable subgroup of $G$. Note that $R$ may not be connected. By the Levi decomposition, we can decompose $G$ as $G=L \ltimes R$, where $L$ is a connected semisimple Lie subgroup of $G$. Since $F$ is centre-free, $R$ is the maximal normal solvable subgroup of $F \times G$. Let $C$ be the compact part of $L$ and $J = CR$. By Auslander's theorem~\cite[Theorem~1 (iii)]{Auslander}, $\Gamma \cap J$ is a uniform lattice of $J$. Recall that $\Gamma \cap G$ is contained in the centre of $G$. Let $G_{1} = (\Gamma \cap J)\backslash G$ and $R_{1} = (\Gamma \cap J)\backslash R$.  Since $R_{1}$ is solvable and compact, its identity component $(R_{1})_{0}$ is abelian. Then the kernel of $R_{1} \to G/[G,G]$ is finite. On the other hand, since $\G$ is a lattice of $F \times L$ and both $F$ and $L$ are non-compact and semisimple, by \cite[Assertion IX.6.18 (B)]{Margulis}, $\G/[\G,\G]$ is finite. Then $R_{1}$ is finite, which implies that $G$ is semisimple.

Let $X' = H'/K'$, where $K' = H \cap K$, and decompose $X$ as $X = X' \times X''$. As in the proof of the case (ii), by replacing $(M,\F)$ with a finite covering, we can assume that $p_{H'}(\Gamma)$ is torsion-free. Then $\pi : \Gamma \backslash (X \times G) \to p_{H'}(\Gamma) \backslash X'$ is a {flat} fibre bundle. Let $\G'' = \G \cap (F \times G)$. The fibre of this bundle is $\G'' \backslash (X'' \times G)$ which admits a $G$-Lie foliation $\mathcal{G}$ whose developing map is the projection $X'' \times G \to G$. By $\Gamma \cap F= \{e\}$, the leaves of $\mathcal{G}$ is simply-connected. 
\end{proof}

We will combine Proposition \ref{dic_cor2} with Theorem \ref{main1_thm} to prove Corollaries \ref{dic_cor}, \ref{z_cor} and \ref{codim_cor} stated in Section \ref{results_sec}. First let us show the following dichotomy in the case where $X$ is irreducible, which is a precise version of Corollary \ref{dic_cor}:

\begin{cor}\label{dic_cor2}
Let $X$ be an \emph{irreducible} symmetric space of non-compact type of
real dimension $n\ge 3$.
 Let $(M,\F)$ be a minimal Lie foliation on a closed manifold. Assume that $\F$ 
bears a $C^{0}$ metric whose restriction to every leaf is
smooth and locally isometric to $X$. Then, one of the following holds.
\begin{enumerate}
\item (Example \ref{hom_ex} with irreducible $\G$) The
Lie group  $G$ is semisimple non-compact,
the fundamental group
 $\pi_{1}M$ is isomorphic to an irreducible uniform
  lattice $\G$ of $\Isom (X) \times G$,
 and $(M,\F)$ is smoothly conjugate to the
  homogeneous Lie foliation on $\G \backslash (X \times G)$ whose lift to $X \times G$ is
  the product foliation parallel to $X$. The leaves of $\F$ are simply-connected.
\item (Example \ref{susp_ex}) The Lie group $G$ admits a compact quotient group $G'$ of the same dimension
such that a finite covering of $(M,\F)$ is smoothly conjugate to a suspension Lie foliation on a $G'$-bundle over a compact manifold locally isometric to $X$. If moreover $X$ is not $\bH_{\RR}^{n}$ nor
 $\bH_{\CC}^{n/2}$, then $G'$ is semisimple.
\end{enumerate}
\end{cor}

In the next, let us explain how to prove Corollary \ref{z_cor}, which is a generalisation of Zimmer's arithmeticity theorem of the holonomy group. Every Lie foliation $(M,\F)$ which satisfies the assumption of Theorem \ref{main1_thm} is homogeneous, in other words, of the form $\G\backslash (X \times G)$ for a uniform lattice $\G$ of $H \times G$, where $H=\Isom X$. By Proposition \ref{lochom_prop}, we can see that, unless $(M,\F)$ is a quotient of a suspension Lie foliation by a finite group and $X$ has an irreducible factor isometric to $\bH_{\RR}^{d}$ or $\bH_{\CC}^{d}$, the adjoint group $\Ad_{G}(G)$ is non-compact and semisimple. Therefore, the arithmeticity of the holonomy group
 $\Ad_{G}(\G)$ in $\Ad_{G}(G)$ makes sense. Since $(\id \times \Ad_{G})(\G)$ is a uniform lattice of $H \times \Ad_{G}(G)$ (Proposition \ref{lochom_prop}), {$\Ad_{G}(\G)$ has a finite index subgroup which is a product of irreducible uniform lattices of either a semisimple Lie group of real rank $\geq 2$, $\Isom(\bH_{\HH}^{d})$ or $\Isom(\bH_{\OO}^{2})$. By using the superrigidity theorems due to Margulis~\cite{Margulis}, Corlette~\cite{Corlette} and Gromov-Schoen~\cite{GS} to follow a well known argument (see \cite[Lemma 9.1]{Zimmer}), we get the arithmeticity of $\Ad_{G}(\G)$ in $\Ad_{G}(G)$, hence Corollary \ref{z_cor}.}
 
Finally let us see how to prove Corollary \ref{codim_cor}, which is a generalisation of \cite[Theorem A-(5)]{Zimmer}. {Since the given foliation is not finitely covered by a suspension foliation, after Theorem \ref{main1_thm} and Proposition \ref{lochom_prop}, there exists a decomposition $X = X' \times X''$ such that $H'' \times \Ad_{G}(G)/C$ admits an irreducible lattice, where $H'' = \Isom(X'')$} and $C$ is the maximal connected normal compact subgroup of $\Ad_{G}(G)$. By Johnson's theorem~\cite{Johnson}, $\Lie(H'' \times \Ad_{G}(G)/C) \otimes \CC$ is a direct sum of mutually isomorphic ideals, which are of dimension $d(X)$. Since the Lie algebra of $\Ad_{G}(G)/C$ contains at least one of these ideals, we obtain
\[\codim (\F) = \dim G = \dim_{\CC} \Lie(\Ad_{G}(G)) \otimes \CC \geq d(X).\]

Apparently, the following example shows that suspension Lie foliations should be excluded from \cite[Theorem A-(5)]{Zimmer}. 

\begin{ex}\label{ex:Z}
Here is a well-known construction coming from Galois conjugation. Let
\begin{align*}
\Gamma & := \SO(x^{2}+y^{2}-\sqrt{2}z^{2}-\sqrt{2}w^{2}; \Z[\sqrt{2}])
\end{align*}
 be a uniform lattice (see e.g., \cite[Section 5.5]{Witte}) in \[H:=\SO(x^{2}+y^{2}-\sqrt{2}z^{2}-\sqrt{2}w^{2}) \cong \SO(2,2).\] 
The automorphism $\sigma$ of the ring $\Z[\sqrt{2}]$ such that
$\sigma(\sqrt{2})=-\sqrt{2}$
 induces a homomorphism from $\Gamma$ into
 \begin{align*}
G & := \SO(x^{2}+y^{2}+\sqrt{2}z^{2}+\sqrt{2}w^{2}) \cong \SO(4)
\end{align*}
  whose image is dense in $G$. Recall that $H$
  has exactly two connected components $H_{\pm} \cong \SO_\pm(2,2)$.
By Selberg's lemma \cite{Selberg}, there exists {in $\Gamma\cap H_+$} a torsion free,
 finite index subgroup $\G'$. Put $V:=\G'\backslash H/K$, where $K$ is a maximal compact subgroup of $H$.
Suspending the quotient map $\pi_1V\to\G'$ composed with $\sigma$, one gets a
   homogeneous $G$-Lie foliation whose leaves are locally isometric to $H/K$.

Let us modify this example to have fibres of smaller dimension. Let $h : \G' \to \SO(3)$ be the composite
\[\G' \to \SO(4) \to \SO(4)/\{\pm I\} \to \SO(3) \times \SO(3) \to \SO(3)\] where the second map is the canonical projection, the third map is the accidental double cover, and the fourth map is the first projection. The discrete subgroup $\G'$ being Zariski dense in $H$,
 the centre of $\G'$ is
contained in the centre of $H$, and consequently, finite.
  Since $\G'$ is torsion-free, it has no centre.
By Margulis' normal subgroup theorem \cite{Margulis}, $h$ is injective.
 We get, by means of $h$, a suspension $\SO(3)$-Lie foliation $\F$, whose leaves
  are simply-connected ($h$ being injective). Every leaf is thus isometric to
   $$\SO(2,2)/\Sr(\Or(2)\times\Or(2))$$
 On the other hand, $$\codim \F = \dim \SO(3) = 3 < 6 = \dim \SO(2,2).$$ 
 Hence, this example does not satisfy the conclusion of \cite[Theorem A-(5)]{Zimmer}.
\end{ex}

The following is an example of homogeneous suspension Lie foliations whose fibers are not semisimple Lie groups. 

\begin{ex}
    Another type of homogeneous suspension $G$-Lie foliation, where $G = S^1 = \R/\Z$,
can be obtained from any closed hyperbolic $n$-manifold $V$ ($n\ge 2$)
whose first Betti number is at least $2$.
The group $\pi_1V$ is a uniform lattice in {$\Isom(\bH_\RR^{n})$}. One
chooses a De Rham cohomology class
 $\omega\in H^{1}(V;\RR)$ not proportional to any
  integral class, and one regards $\omega$
 as a morphism $\pi_1V\to\RR$. After a rescaling, the image of $\omega$ in $\RR$
 contains $\Z$. Hence, $\omega$ induces a morphism  $\rho:\pi_1V\to S^1$
  whose image is dense in $S^1$.  
\end{ex}

\section{Application to Riemannian foliations}

Finally, we apply the rigidity on Lie foliations,  Theorem~\ref{main1_thm}, to show a rigidity result on Riemannian foliations. Recall that, by Molino's theory \cite{Molino}, every minimal Riemannian foliation $(M,\F)$ is obtained as the quotient of a minimal $G$-Lie foliation $(M^{1},\F^{1})$ of the same dimension by a free action of a compact Lie group $S$, where $G$ is called the \emph{structural group} of $(M,\F)$.

\begin{thm}\label{uni_thm}
Let $(M,\F)$ be a connected closed manifold with a minimal Riemannian foliation. Let $G$ be the structural group of $(M,\F)$. Assume that $M$ admits a Riemannian metric of class $C^0$ whose restriction to every leaf of $\F$ is smooth and locally isometric with a symmetric space $X$ of non-compact type with no Poincar\'{e} disk factor. Then, there exist a compact Lie subgroup $S$ of $G$ and a uniform lattice $\G$ of $\Isom (X) \times G$ such that a finite covering of $(M,\F)$ is smoothly conjugate to $(\Gamma \backslash (X \times (G/S)), \F_{0})$, where $\F_{0}$ is the foliation on $\Gamma \backslash (X \times (G/S))$ whose lift to $X \times (G/S)$ is $\sqcup_{g \in G} X \times gS$.
\end{thm}

Let us outline the proof. 
Consider the principal $S$-bundle $\pi : (M^1,\F^1) \to (M,\F)$ obtained by Molino's theory, where $(M^1,\F^1)$ is a minimal Lie foliation and $S$ is a compact Lie group. By taking a finite covering of $(M,\F)$, we can assume that $S$ is connected. Each leaf of $\F^{1}$ covers a leaf of $\F$ by the restriction of $\pi$. Therefore, if $(M,\F)$ admits a leafwise locally symmetric Riemannian metric, then $(M^{1},\F^{1})$ admits such leafwise metric that is invariant under the principal $S$-action. Thus Theorem \ref{uni_thm} follows from Theorem \ref{main1_thm} and the following description of the leafwise isometry groups of homogeneous Lie foliations.

\begin{lem}\label{isom_lem}
Let us use the notation of Example \ref{hom_ex}. Let $(M,\F)$ be the homogeneous Lie foliation considered there, namely, $M = \Gamma \backslash (X \times G)$ and $\F$ is the foliation on $M$ whose lift to $X \times G$ is $\sqcup_{g \in G} X \times \{g\}$. Assume that $\F$ is minimal and $G$ is simply-connected. Let $\Isom (M,\F)$ be the group of diffeomorphisms which map each leaf of $\F$ to a leaf of $\F$ preserving the leafwise metric. Its identity component $\Isom_{0}(M,\F)$ is given by
\begin{equation}
\Isom_{0}(M,\F) = (G \cap \Gamma) \backslash G,
\end{equation}
where $G \cap \Gamma$ is normal in $G$ and $g_{1} \in (G \cap \Gamma) \backslash G$ acts on $M$ by $g_{1} \cdot (\Gamma(x, g)) = \Gamma (x, gg^{-1}_{1})$. 
\end{lem}

\begin{proof}
Only the inclusion $\Isom_{0}(M,\F) \subset (G \cap \Gamma) \backslash G$ may not seem immediate. In order to verify it, recall that, since $(M,\F)$ is a minimal Lie foliation,
    the differential map of any $f \in \Isom(M,\F)$ induces an automorphism
     of the Lie algebra of
the \emph{transverse}
 vector fields on $(M,\F)$ (in the sense of Molino \cite{Molino}), which is $\Lie (G)$.
  In view of this fact, it is not difficult to see that, for any lift $\tilde{f} \in \Diff (X \times G)$ of $f$, there exist $h_{1} \in \Isom (X)$, $g_{1} \in G$ and $\alpha \in \Aut (G)$ such that
\[
\wt{f}(x, g) = (h_{1}x, \alpha(g) g_{1}), \quad\quad \forall x \in X, \forall g \in G
\]
(see Haefliger \cite{Haefliger}, see also \cite[Proposition 6]{Nozawa}). The inclusion easily follow.
\end{proof}

\section{Problems}\label{prob_sec}

With this work, we hope to exemplify the use, for Lie foliations, of rigidity theory; and also to open and motivate some fundamental problems in the classification of Lie foliations. 
(In these problems, every Lie foliation is
 understood to be a minimal $G$-Lie foliation on a closed manifold; and $X$
is understood to be a product of irreducible Riemannian symmetric spaces of non-compact type).

\begin{prob}\label{que:1}
Classify Lie foliations of leafwise dimension two. In particular, if the leaves
of a Lie foliation are \emph{isometric to the Poincar\'e disk,} is this foliation necessarily homogeneous?
\end{prob} 
This last question, of particular interest for the theory of Riemannian
foliations, is open
 even for $G=\PSL(2;\RR)$.
 
Remarkably few is known in the direction of Problem \ref{que:1}. Matsumoto-Tsuchiya
proved that for  $G=\GA(1,\R)$, the real $2$-dimensional nonabelian
solvable group, every $G$-Lie foliation (minimal or not)
 on a closed $4$-manifold is homogeneous \cite{MT}. The already cited example built in \cite{HMM}, where
$G=\PSL(2;\RR)$ and the leaves are the $2$-sphere minus a Cantor set,
is not homogeneous;  the holonomy group $\Gamma$ is not a cocompact lattice in any real Lie group; but there
is nevertheless a kind of rigidity:  $\Gamma$ is a cocompact lattice in the
product of $G$ with the diadic Lie group $\PSL(2;\Q_2)$.

\medbreak
\begin{prob}
Assume that $X$ has no factor $\bH^2_{\RR}$, and 
that $\F$ is a Lie foliation in which the
 universal cover of every leaf
 is \emph{bi-Lipschitz diffeomorphic} to $X$. Does it follow that $\F$ is homogeneous?
\end{prob}

\medbreak
\begin{prob} Would Theorem \ref{main2_thm} still hold if one allowed
$X$ to have some
 factors $\bH^d_{\RR}$ ($d\ge 3$) or $\bH^d_{\CC}$ ($d\ge 2$)?
\end{prob}

\begin{prob}
Which $3$-manifolds can be realized as the leaf of a Lie foliation?
\end{prob}

Schweitzer-Souza~\cite{SS} found a necessary condition for
a non-compact Riemannian manifold to be a leaf of a Riemannian foliation on a closed manifold.

\begin{prob}\label{que:2} 
Is there a geometrisation theorem for Lie foliations of dimension $3$?
\end{prob}

Such a foliated version of Perelman's geometrisation theorem would split every $3$-dimensional Lie foliation, along some transverse closed hypersurfaces,
into pieces, in each of which the leaves
 would admit one of Thurston's eight geometries.

\begin{prob}\label{que:3} 
Classify $3$-dimensional Lie foliations whose leaves admit one of Thurs\-ton's eight geometries.
\end{prob}

The present paper solves the hyperbolic case. This problem enters the more general program, proposed by Paul Schweitzer,
 to study $3$-dimensional foliations with geometric leaves. For example, it is not known if there is a non-homogeneous Lie foliation $\F$ whose leaves are isometric to the $3$-dimensional solvable Lie group $\operatorname{Sol}$. (Recall that after Carri\`ere \cite{Carriere}, $G$ would necessarily also be solvable).

{Ghys \cite{Ghys} pointed out that an equicontinuous foliated space is a topological version of Riemannian foliations. The Molino theory for such spaces was estalished by the works of
 Alvarez-L\'opez, Barral-Lij\'o, Candel, Clark, Dyer,
 Hurder, Lukina, Moreira Galicia
 \cite{Alvarez09,Alvarez16,Alvarez19,Dyer17}.
A foliated space can be transversely modelled on the Cantor space, and it can be regarded as a model of minimal sets of foliations.}
\begin{prob}
    Study the rigidity of equicontinuous foliated spaces with locally symmetric leaves.
\end{prob}

\end{document}